\documentclass[a4paper,11pt,reqno]{amsart}

\usepackage[margin=1in]{geometry}
\usepackage{amsmath,amssymb,amsthm,mathrsfs,mathtools}
\usepackage{enumitem}
\usepackage{hyperref}
\usepackage{xurl}
\usepackage{cleveref}

\usepackage{xcolor}

\usepackage{tikz}
\usepackage{eso-pic}

\usepackage{microtype}
\usepackage[T1]{fontenc}

\hypersetup{colorlinks=true,linkcolor=blue,citecolor=blue,urlcolor=blue}

\newtheorem{theorem}{Theorem}
\newtheorem{lemma}[theorem]{Lemma}

\theoremstyle{definition}
\newtheorem{definition}[theorem]{Definition}

\theoremstyle{remark}
\newtheorem{remark}{Remark}

\DeclareMathOperator{\CWP}{CWP}
\DeclareMathOperator{\val}{val}
\newcommand{\Free}{\mathbb F}
\newcommand{\PL}{\mathrm{PL}}

\newcommand{\defp}{\operatorname{def}_{+}}
\newcommand{\wedgek}[1]{\bigvee_{#1} S^2}

\newcommand{\tri}{\mathcal{T}}
\newcommand{\lk}{\operatorname{lk}}
\newcommand{\cs}{\mathbin{\#}}
\newcommand{\bs}{\mathbin{\natural}}

\title{PL Recognition After Two $S^{2}\times S^{2}$ Stabilisations is PSPACE-Hard}
\author{Rhuaidi Antonio Burke}
\address{Mathematical Institute, University of Oxford, Oxford OX2 6GG, United Kingdom}
\email{rhuaidi.burke@maths.ox.ac.uk}
\date{\today}

\begin{document}
	
	\begin{abstract}
		Let $Z$ be any fixed closed connected PL $4$-manifold. We give a
		polynomial-time many-to-one reduction from the compressed word problem in
		Thompson's group $F$ to fixed-target recognition of
		$Z\cs_{2}(S^{2}\times S^{2})$. 
		Once a finite triangulation of $Z$ has been fixed, the construction sends a straight-line program $A$ to a closed triangulated PL $4$-manifold $X_{A,Z}$ such that
		\[
		\val(A)=1\text{ in }F
		\quad\Longleftrightarrow\quad
		X_{A,Z}\cong_{\PL}Z\cs_{2}(S^{2}\times S^{2}).
		\]
		Since the compressed word problem in $F$ is PSPACE-complete, every such
		recognition problem is PSPACE-hard. 
		In particular, fixed-target recognition of $\cs_n(S^2\times S^2)$ is
		PSPACE-hard for every fixed $n\geq2$.
	\end{abstract}
	
	\vspace*{-0.5em}
	\maketitle
	
	\section{Introduction}
	\label{sec:intro}	
	One of the fundamental problems in topology is to determine whether two manifolds are equivalent in a specified category, such as the topological, smooth, or piecewise linear (PL) category. 
	For a fixed closed PL $n$-manifold $X$, the \emph{PL recognition problem} asks, given a finite PL triangulation of a closed $n$-manifold $Q$, whether $Q$ is PL-homeomorphic to $X$. 
	If no algorithm solves this problem, we say that $X$ is \emph{PL-unrecognisable}. 
	The \emph{topological recognition problem} uses the same input class and asks whether $Q$ is homeomorphic to $X$; if no algorithm solves this problem, we say that $X$ is \emph{topologically unrecognisable}. 
	The decidability and the complexity of recognition depend strongly on the dimension, the chosen target, and the category of equivalence.
	
	\begin{remark}
		In dimensions at most three, every topological manifold admits essentially unique PL and smooth structures. Consequently, for such manifolds the notions of homeomorphism, PL-homeomorphism, and diffeomorphism coincide. Thus, in dimensions at most three, there is no need to
		distinguish between the corresponding homeomorphism, PL-homeomorphism, and
		diffeomorphism recognition problems; it suffices to consider the homeomorphism problem.
		Moreover, every PL manifold of dimension at most six admits a unique compatible smoothing up to isotopy \cite{HirschMazur,Munkres-Smoothing}. We may therefore use smooth and PL arguments interchangeably for the $4$- and $5$-manifolds appearing below.
	\end{remark}
	
	In dimensions one and two, the classification of manifolds gives direct
	recognition and homeomorphism algorithms. For closed orientable 3-manifolds, Kuperberg \cite{Kuperberg} gives a homeomorphism algorithm whose running time is bounded by a tower of exponentials of bounded height.
	
	The situation changes in dimension four. Markov \cite{Markov1960} proved that for every
	$n\geq 4$ there is a closed $n$-manifold whose recognition problem is
	undecidable. In particular, the general homeomorphism problem
	is undecidable in every dimension $n\geq 4$. In the topological category, Shtan'ko \cite{Shtanko} and
	Chernavsky--Leksine \cite{ChernavskyLeksine} established that $\cs_{14}(S^2\times S^2)$ is unrecognisable. 
	Gordon \cite{Gordon} later reduced the number of summands to twelve, with Tancer \cite{Tancer} then reducing it to nine. More recently, Kegel--Li--Ren \cite{KegelLiRen2026} have further reduced the number of summands to seven. In the smooth category, Gordon's construction establishes the smooth unrecognisability of $\cs_{12}(S^2\times S^2)$, though he stated his result in the context of the topological category. In the same work of Kegel--Li--Ren, the smooth unrecognisability of $\cs_9(S^2\times S^2)$ is also established. 
	
	Even for the $n$-sphere there is a sharp division between the known cases. Recognition of $S^n$ is elementary for $n\leq 2$. Rubinstein \cite{Rubinstein} and Thompson \cite{Thompson} proved that $S^3$ is recognisable, with Schleimer \cite{Schleimer-s3np} and, independently, Ivanov \cite{Ivanov-s3np} later showing that $3$-sphere recognition lies in the complexity class $\mathrm{NP}$. 
	On the other hand, Novikov \cite[Section 10]{VKF} proved that recognition of $S^n$ is
	undecidable for every $n\geq 5$. The remaining case is
	dimension four, where it is currently not known whether the standard PL $4$-sphere can be
	recognised algorithmically.
	
	\begin{remark}
		The undecidability of the general homeomorphism problem in dimension four does
		not settle the recognition problem for a prescribed $4$-manifold. These are
		different decision problems. In the general problem, both input manifolds
		vary. For a fixed manifold $X$, by contrast, $X$-recognition asks whether a
		single input manifold is PL-homeomorphic to that particular $X$.
		
		Thus general undecidability rules out one algorithm that compares every pair
		of input $4$-manifolds. It does not imply that $X$-recognition is undecidable
		for every $X$, or for any particular $X$ chosen in advance. In principle,
		each fixed manifold could admit its own recognition algorithm without there
		being a uniform procedure that obtains or combines these algorithms to solve
		the general problem. General undecidability therefore leaves open the
		recognisability of specific manifolds.
	\end{remark}
	
	The results above leave open what can be proved for small stabilisations of $S^4$. Our main result shows that fixed-target PL recognition of
	\[
	\cs_2(S^2\times S^2)
	\]
	among closed triangulated PL $4$-manifolds is PSPACE-hard under polynomial-time many-to-one reductions. More generally, this hardness persists after taking connected sum with any fixed closed connected PL $4$-manifold---for every such manifold $Z$, recognition of
	\[
	Z\cs_2(S^2\times S^2)
	\]
	is PSPACE-hard. Here $Z$ is fixed in advance and is not part of the input. Taking $Z=\cs_{n-2}(S^2\times S^2)$, with $Z=S^4$ when $n=2$, shows that fixed-target recognition of $\cs_n(S^2\times S^2)$ is PSPACE-hard for every fixed $n\geq2$.
	
	The reduction starts from the compressed word problem for Thompson’s group $F$, which has recently been shown to be PSPACE-complete \cite{BFLaW}. An instance is a finite list $A$ of instructions for building a word (a so-called \emph{straight-line program}). Each instruction either supplies a generator or its inverse, or concatenates two words produced by earlier instructions. One instruction is designated as the output and produces the word $\val(A)$. From $A$ we construct in polynomial time a finite presentation $P_A$ and then a closed triangulated PL $4$-manifold $X_A=X(P_A)$. If $\val(A)=1$ in $F$, the relator tuple of $P_A$ can be transformed by Andrews--Curtis moves into a tuple consisting of its generators followed by two empty relators; a thickening argument then identifies $X_A$ with $\cs_2(S^2\times S^2)$. If $\val(A)\ne1$, then $F$ embeds in the group presented by $P_A$, denoted $G(P_A)$, while $\pi_1(X_A)\cong G(P_A)$.
	
	Fixing a triangulation of $Z$, we form $X_{A,Z}:=X_A\cs Z$. In a yes-instance this is the required target. In a no-instance, the Seifert--van Kampen and Grushko--Neumann theorems give
	\[
	\pi_1(X_{A,Z})\cong G(P_A)*\pi_1(Z),
	\qquad
	d\bigl(\pi_1(X_{A,Z})\bigr)=d(G(P_A))+d(\pi_1(Z))>d(\pi_1(Z)),
	\]
	where $d(H)$ denotes the minimum number of generators of a finitely generated group $H$. The target has fundamental group $\pi_1(Z)$, so the no-instance cannot be PL-homeomorphic to it.
	
	A central feature of the construction is that the program $A$ is used directly---the potentially exponentially long word $\val(A)$ is never expanded. The presentation is converted into a polynomial-size Kirby diagram for the $2$-handlebody associated to $P_A$, and then into a polynomial-size triangulation of the double of this handlebody. The final connected sum with the fixed manifold $Z$ has constant overhead. Both yes- and no-instances therefore produce triangulations of closed PL $4$-manifolds.
	
	\subsection*{Organisation of the paper}
	\Cref{sec:background} fixes conventions and introduces surplus Andrews–Curtis triviality.
	\Cref{sec:main-statements} states the fixed-target recognition problem and the main theorem precisely.
	\Cref{sec:alg-input,sec:miller-tancer-gadgetry,sec:final-presentation} construct and analyse the presentation $P_A$. \Cref{sec:triangulation} converts finite presentations into triangulated closed $4$-manifolds, and \Cref{sec:output-topology} identifies their fundamental groups and their PL types in the Andrews--Curtis-trivial case. 
	\Cref{sec:proof-of-theorem} proves the main theorem. 
	Finally, \Cref{sec:final-remarks} explains the structural obstacles to
	reaching the lower simply connected targets through the direct
	presentation--thickening pipeline presented in the paper.
		
	\subsection*{Acknowledgements}
	The author would like to thank Marc Lackenby, Tobias Hirsch, Christopher Douglas, and Maria Stuebner for reading early drafts of this paper and for their many helpful comments and suggestions.
	This work was supported by the EPSRC grant EP/Y016270/1, ``Computable manifolds''.
	
	\subsection*{Statement on LLM use}
	ChatGPT-5.6 Sol Pro was used to assist with algebraic calculations and to act as an adversarial reviewer during the preparation of this paper.

	\section{Preliminaries}
	\label{sec:background}
	In this section we fix the main conventions used in the paper. 
	
	\subsection*{Triangulations}
	Throughout, a \emph{triangulation} $\tri$ will refer to a finite collection of $n$ abstract $d$-simplices, some or all of whose facets are affinely identified in pairs. 
	We allow facets of the same $d$-simplex to be identified. The gluings defining $\tri$ also have the effect of merging the lower-dimensional faces of the simplices into equivalence classes, which we refer to as the \emph{faces} of $\tri$. We say that $\tri$ is \emph{valid} if the gluings never identify a face with itself along a non-identity~map.
	
	The \emph{link} $\lk(v)$ of a vertex $v$ of $\tri$ is the frontier of a small regular neighbourhood of $v$. We treat vertex links as triangulated $(d-1)$-dimensional spaces, formed by inserting a small $(d-1)$-simplex into each corner of each $d$-simplex, and then joining together the $(d-1)$-simplices from adjacent $d$-simplices along their facets. 
	
	Throughout, we restrict attention to valid triangulations whose vertex links are either PL $(d-1)$-balls or closed PL $(d-1)$-manifolds.
	A vertex is called \emph{finite} if its link is a PL sphere, \emph{boundary} if its link is a PL ball, and \emph{ideal} if its link is a closed PL manifold that is not a PL sphere.
	If $\tri$ has at least one ideal vertex, we call $\tri$ an {\em ideal triangulation}. We think of an ideal vertex as not being part of $\tri$, but instead representing a boundary component homeomorphic to the vertex link. We say that such boundary components form the \emph{ideal boundary} of $\tri$. If all vertex links of $\tri$ are piecewise linear triangulations of spheres, then the triangulation $\tri$ is a \emph{piecewise linear} triangulation of a \emph{closed} manifold. If all vertex links of $\tri$ are PL triangulated spheres or balls, $\tri$ is a PL triangulation of a manifold with \emph{real boundary}, where the real boundary is given by the unglued facets in the triangulation. If $\tri$ is a PL triangulation, with or without real boundary, then the underlying set, or the {\em carrier}, of $\tri$ is a PL manifold denoted by $|\tri|$. For ideal triangulations $\tri$, the carrier denotes the underlying set of $\tri$ with ideal vertices removed.
	
	We write $f_i(\tri)$ to denote the number of $i$-dimensional simplices in $\tri$.

	\subsection*{Spines, thickenings and 3-deformations} 
	A \emph{spine} $K$ of a compact PL manifold $X$ is a finite subpolyhedron $K\subset\operatorname{int}(X)$ such that $X$ collapses onto $K$, denoted $X\searrow K$, where collapses are understood after suitable triangulation and subdivision. Equivalently, $X$ is a regular neighbourhood of $K$. A thickening of a finite complex $K$ is a compact manifold equipped with a PL embedding of $K$ as a spine; we identify $K$ with its embedded image when convenient. For smooth manifolds, these conditions refer to the compatible PL structure. A \emph{$3$-deformation} is a finite sequence of elementary expansions and collapses through complexes of dimension at most three. We refer the reader to \cite{Matveev-3Manifolds} for more details.
	
	\subsection*{Andrews--Curtis conventions}
	\label{sec:ac-setup}
	
	\begin{definition}[Surplus Andrews--Curtis triviality]
		Let $X=\{x_1,\dots,x_n\}$, $d\geq 0$, and let
		\[
		P=\langle X\mid r_1,\dots,r_{n+d}\rangle .
		\]
		We say that $P$ is $d$-Andrews--Curtis trivial, or $d$-AC-trivial, if the
		$(n+d)$-tuple $(r_1,\dots,r_{n+d})$ in the free group $\Free(X)$ is
		Andrews--Curtis equivalent to
		\[
		(x_1,\dots,x_n,1,\dots,1),
		\]
		with $d$ trailing empty relators.
		
		The Andrews--Curtis moves allowed are
		\[
		r_i\mapsto r_i^{-1},\qquad
		r_i\mapsto r_i r_j^{\pm1}\quad(i\neq j),\qquad
		r_i\mapsto u r_i u^{-1}\quad(u\in \Free(X)),
		\]
		together with permutation of relators.
	\end{definition}
	
	\begin{lemma}
		\label{lem:ambient-normal-closure}
		Let $Y$ be a finite generating set, and let $T=(t_1,\dots,t_q)$ be a tuple of
		relators in $\Free(Y)$. Suppose $b_1,\dots,b_k$ are entries of $T$, and let
		$a$ be another entry of $T$. If
		\[
		U\in\left\langle\!\left\langle b_1,\dots,b_k\right\rangle\!\right\rangle_{\Free(Y)},
		\]
		then $T$ is Andrews--Curtis equivalent to the tuple obtained from $T$ by
		replacing $a$ by $aU$, while leaving every other entry unchanged.
	\end{lemma}
	
	\begin{proof}
		Write
		\[
		U=\prod_{\ell=1}^{N}u_\ell b_{j_\ell}^{\varepsilon_\ell}u_\ell^{-1},
		\qquad u_\ell\in\Free(Y),\quad \varepsilon_\ell\in\{-1,1\}.
		\]
		For a single factor $u b_j^{\varepsilon}u^{-1}$, first invert $b_j$ if
		$\varepsilon=-1$, then conjugate $b_j^{\varepsilon}$ by $u$, multiply
		$a$ by the resulting relator, and undo the conjugation and inversion of
		$b_j$. This changes only $a$. Repeating this procedure for all factors
		replaces $a$ by $aU$ and restores every $b_j$ to its original value.
	\end{proof}
	
	\subsection*{Straight-line programs and the compressed word problem}
	A \emph{straight-line program} over an alphabet $\Sigma$ is a
	finite list of word-building instructions with one distinguished output instruction. We use
	the standard normal form in which each gate is either labelled by a
	letter of $\Sigma$ or is the ordered product of two earlier gates.
	A general \emph{straight-line program}, including one with longer right-hand sides, can be put in this form in linear time. In the group-theoretic applications below, $\Sigma$ is a fixed nonempty alphabet closed under inversion. Consequently, an empty-word gate may be replaced, in constant size, by a subprogram representing $aa^{-1}$, where $a\in\Sigma$ is fixed. This replacement preserves the represented group element, although not the formal word. The size of a straight-line program is the number of gates.
	The word $\val(A)$ represented by the output gate can have length exponential in this size.
	
	For a finitely generated group $G$, the \emph{compressed word problem}
	$\CWP(G)$ asks whether $\val(A)=1$ in $G$. All constructions below use
	the circuit directly and never expand $\val(A)$.
	
	\subsection*{Complexity classes}	
	The class $\mathrm{PSPACE}$ consists of decision problems solvable using space polynomial in the input size. Every polynomial-time computation uses polynomial space, and a nondeterministic polynomial-time computation tree can be searched depth-first using polynomial space. Hence
	\[
	\mathrm{P}\subseteq\mathrm{NP}\subseteq\mathrm{PSPACE}.
	\]
	Consequently, every problem that is $\mathrm{PSPACE}$-hard under
	polynomial-time many-to-one reductions is also $\mathrm{NP}$-hard. 
	Note that this does not imply that the problem itself belongs to $\mathrm{PSPACE}$.
	
	\subsection*{Rank and free products}
	For a finitely generated group $G$, write $d(G)$ for the minimum
	cardinality of a generating set. The Grushko--Neumann
	theorem~\cite{Neumann1943} states that if $G$ and $H$ are finitely generated,
	then
	\[
	d(G*H)=d(G)+d(H).
	\]
	
	\section{Main statements}
	\label{sec:main-statements}
	
	For a finite presentation
	\(
	P=\langle x_1,\ldots,x_n\mid r_1,\ldots,r_m\rangle,
	\)
	we write $G(P)$ for the group presented by $P$, and define its relator surplus by
	\[
	\operatorname{def}_+(P):=m-n.
	\]
	Thus $\operatorname{def}_+(P)=k$ means that $P$ has $k$ more relators than generators. This convention is the negative of the usual presentation deficiency.
	
	Let $X$ be a fixed closed connected PL $4$-manifold.
	The fixed-target PL recognition problem for $X$ has as input a finite triangulation $\mathcal{T}$, promised to represent a closed PL $4$-manifold, and asks whether
	\(
	|\mathcal{T}|\cong_{\mathrm{PL}} X.
	\)
	
	Hardness is understood with respect to reductions that always satisfy this promise. More precisely, a polynomial-time many-to-one reduction from a decision problem $\Pi$ is a polynomial-time construction assigning to every instance $A$ of $\Pi$ a finite triangulation $\mathcal{T}_A$ such that $|\mathcal{T}_A|$ is a closed PL $4$-manifold and
	$A$ is a yes-instance of $\Pi$ if and only if $|\mathcal{T}_A|\cong_{\PL} X$.
	In particular, the outputs corresponding to both yes- and no-instances are triangulations of closed PL $4$-manifolds.
	
	\begin{theorem}
		\label{thm:main}
		Let $Z$ be any fixed closed connected PL $4$-manifold. Fixed-target
		PL recognition of
		\[
		Z\cs_2(S^2\times S^2)
		\]
		among closed triangulated PL $4$-manifolds is PSPACE-hard under
		polynomial-time many-to-one reductions satisfying the preceding promise.
		More precisely, fix a finite triangulation of $Z$. There is a polynomial-time algorithm that, given a straight-line program $A$ over the standard generators of Thompson’s group $F$, constructs a triangulation
		$\mathcal{T}_{A,Z}$ of a closed PL $4$-manifold such that
		\[
		\val(A)=1\text{ in }F
		\quad\Longleftrightarrow\quad
		|\mathcal{T}_{A,Z}|\cong_{\mathrm{PL}}
		Z\cs_2(S^2\times S^2).
		\]
		Consequently, for every fixed $n\geq2$, fixed-target PL recognition of
		$\cs_n(S^2\times S^2)$ is PSPACE-hard. All these recognition problems
		are also NP-hard.
	\end{theorem}
	
	\section{The algebraic input from Thompson's group \texorpdfstring{$F$}{F}}
	\label{sec:alg-input}
	We use Thompson's group $F$ in the standard balanced finite presentation
	\[
	F=\langle x_0,x_1\mid R_1,R_2\rangle,
	\]
	where
	\[
	R_1=[x_0x_1^{-1},x_0^{-1}x_1x_0],
	\qquad
	R_2=[x_0x_1^{-1},x_0^{-2}x_1x_0^2].
	\]
	
	The compressed word problem $\CWP(F)$ asks, given a straight-line program $A$
	over $x_0^{\pm1},x_1^{\pm1}$, whether the word $\val(A)$ represented by
	$A$ is equal to $1\in F$. Bartholdi--Figelius--Lohrey--Wei\ss{} have shown that
	$\CWP(F)$ is PSPACE-complete \cite{BFLaW}. Thus, to prove
	PSPACE-hardness of the recognition problem, it suffices to give a
	polynomial-time many-to-one reduction from $\CWP(F)$.
	
	\subsection{From a compressed word to a balanced presentation of \texorpdfstring{$F$}{F}}
	\label{sec:cw-to-pres}
	
	Let $A$ be a straight-line program over $x_0^{\pm1},x_1^{\pm1}$. Write the gates, in topological order, as $Y_1,\ldots,Y_N$, so that every multiplication gate uses only earlier gates. For each gate
	$Y_i$, introduce a generator $y_i$. If $Y_i=Y_jY_k$, add the relator
	\[
	d_i=y_i^{-1}y_jy_k.
	\]
	If $Y_i=x_\ell^\varepsilon$, where $\ell\in\{0,1\}$ and
	$\varepsilon\in\{-1,1\}$, add the relator
	\[
	d_i=y_i^{-1}x_\ell^\varepsilon.
	\]
	Let $y_{\mathrm{out}}$ denote the output generator, and set
	\[
	Q_A=
	\left\langle
	x_0,x_1,y_1,\dots,y_N
	\ \middle|\
	R_1,R_2,d_1,\dots,d_N
	\right\rangle .
	\]
	
	\begin{lemma}\label{lem:QA}
		The presentation $Q_A$ is a balanced Tietze extension of the standard
		presentation of $F$. Moreover,
		\[
		y_{\mathrm{out}}=1\text{ in }G(Q_A)
		\quad\Longleftrightarrow\quad
		\val(A)=1\text{ in }F.
		\]
		Finally,
		\[
		Q_A^+=
		\left\langle
		x_0,x_1,y_1,\dots,y_N
		\ \middle|\
		R_1,R_2,d_1,\dots,d_N,x_0,x_1
		\right\rangle
		\]
		is $2$-AC-trivial.
	\end{lemma}
	
	\begin{proof}
		Each relator $d_i$ is a definitional Tietze relator identifying $y_i$ with
		the word computed by the $i$-th gate. Therefore $Q_A$ presents $F$, and
		$y_{\mathrm{out}}$ represents $\val(A)$. It has $N+2$ generators and
		$N+2$ relators, so it is balanced.
		
		It remains to prove $2$-AC-triviality of $Q_A^+$. In the relator tuple
		\[
		(R_1,R_2,d_1,\dots,d_N,x_0,x_1),
		\]
		move $x_0,x_1$ to the front. Using Lemma~\ref{lem:ambient-normal-closure},
		delete all occurrences of $x_0^{\pm1}$ and $x_1^{\pm1}$ from the other
		relators. Thus $R_1$ and $R_2$ become empty relators.
		
		Now process the gate relators in increasing order. If
		$d_i=y_i^{-1}x_\ell^\varepsilon$, then after $x_\ell$ has been eliminated,
		$d_i$ becomes $y_i^{-1}$, which may then be inverted to $y_i$. If
		$d_i=y_i^{-1}y_jy_k$, then $j,k<i$, so $y_j$ and $y_k$ already occur as single-letter relators; another application of Lemma~\ref{lem:ambient-normal-closure}
		turns $d_i$ into $y_i^{-1}$, and then into $y_i$. Hence the relator tuple
		is AC-equivalent to
		\[
		(x_0,x_1,y_1,\dots,y_N,1,1),
		\]
		which is precisely $2$-AC-triviality.
	\end{proof}
	
	\section{The Miller--Tancer gadget}
	\label{sec:miller-tancer-gadgetry}
	
	Let
	\(
	P=\langle x_1,\dots,x_n\mid r_1,\dots,r_m\rangle
	\)
	be a finite presentation, let $w\in\Free(x_1,\dots,x_n)$, and let
	$S=(s_1,\dots,s_{m'})$ be a tuple of additional words in
	$\Free(x_1,\dots,x_n)$. Write
	\[
	P^+=\langle x_1,\dots,x_n\mid
	r_1,\dots,r_m,s_1,\dots,s_{m'}\rangle
	\]
	for the presentation obtained by adjoining them.
	
	The following gadget is a modification of the constructions of Miller and
	Tancer \cite{Miller1992,Tancer}. We retain Miller's commutator in the third
	relation and use symmetric conjugating powers in the fourth family. 
	These choices allow the unrestricted embedding argument given below.
	Define $M(P,w,S)$ to have generators
	\[
	x_1,\dots,x_n,\alpha,\beta,\gamma,
	\]
	the relators $r_1,\dots,r_m$, and the following additional relators:
	\[
	\alpha^{-1}\beta\alpha
	=
	\gamma^{-1}\beta^{-1}\gamma\beta\gamma, \tag{MT1}
	\]
	\[
	\alpha^{-2}\beta^{-1}\alpha\beta\alpha^{2}
	=
	\gamma^{-2}\beta^{-1}\gamma\beta\gamma^{2}, \tag{MT2}
	\]
	\[
	\alpha^{-3}[w,\beta]\alpha^{3}
	=
	\gamma^{-3}\beta\gamma^{3}, \tag{MT3}
	\]
	and, for $i=1,\dots,m'$,
	\[
	\alpha^{-(i+3)}s_i\beta\alpha^{i+3}
	=
	\gamma^{-(i+3)}\beta\gamma^{i+3}. \tag{MT4}
	\]
	Each equation contributes its left-hand side times the inverse of its
	right-hand side as a relator. The presentation has $n+3$ generators and
	$m+m'+3$ relators, and therefore
	\[
	\defp(M(P,w,S))=m+m'-n.
	\]
	
	\begin{lemma}
		\label{lem:branch-normal-form}
		Let $A$ be a group and $H=A*\langle t\rangle$. Let
		$D\subset\mathbb Z$ be finite. For each $d\in D$, choose exactly one
		element $z_d$ of one of the two forms
		\[
		z_d=t^{-d}a_dt^d,
		\qquad a_d\in A\text{ of infinite order}, \tag{P$_d$}
		\]
		or
		\[
		z_d=t^{-d}b_d^{-1}tb_dt^d,
		\qquad b_d\in A\setminus\{1\}. \tag{Q$_d$}
		\]
		Then $\{z_d:d\in D\}$ is a free basis for the subgroup it generates in $H$.
	\end{lemma}
	
	\begin{proof}
		Consider a nonempty reduced word in the displayed generators, written after
		collecting adjacent powers as
		\[
		z_{d_1}^{k_1}\cdots z_{d_r}^{k_r},
		\qquad
		k_j\ne0,\quad d_j\ne d_{j+1}.
		\]
		In case $(\mathrm P_d)$,
		\[
		z_d^k=t^{-d}a_d^kt^d,
		\]
		where $a_d^k\ne1$. In case $(\mathrm Q_d)$,
		\[
		z_d^k=t^{-d}b_d^{-1}t^kb_dt^d.
		\]
		Thus, for each $j$, there is a nontrivial reduced word $c_j$ which
		begins and ends with a nontrivial $A$-syllable such that
		\(
		z_{d_j}^{k_j}=t^{-d_j}c_jt^{d_j}.
		\)
		Expanding the product gives
		\(
		t^{-d_1}c_1t^{d_1-d_2}c_2
		\cdots
		t^{d_{r-1}-d_r}c_rt^{d_r}.
		\)
		Every internal power $t^{d_j-d_{j+1}}$ is nontrivial. Hence this is a
		nonempty reduced normal form in $A*\langle t\rangle$. No nonempty reduced
		word in the $z_d$'s is trivial, proving the assertion.
	\end{proof}
	
	\begin{theorem}
		\label{thm:miller-tancer}
		For arbitrary $P,w,S$, the following hold.
		\begin{enumerate}[label=\textup{(\roman*)}]
			\item If $w=1$ in $G(P)$, then the map induced by the old generators is
			an isomorphism
			\[
			G(P^+)\xrightarrow{\ \cong\ }G(M(P,w,S)).
			\]
			\item If $w\ne1$ in $G(P)$, then the natural homomorphism
			\[
			G(P)\longrightarrow G(M(P,w,S))
			\]
			is injective.
		\end{enumerate}
		Consequently, if $P^+$ presents the trivial group, then
		\[
		M(P,w,S)\text{ presents the trivial group}
		\quad\Longleftrightarrow\quad
		w=1\text{ in }G(P).
		\]
	\end{theorem}
	
	\begin{proof}
		Suppose first that $w=1$ in $G(P)$. Then
		$[w,\beta]=1$ after imposing the relators of $P$, so (MT3) gives
		$\gamma^{-3}\beta\gamma^3=1$ and hence $\beta=1$. Relation (MT1)
		then gives $\gamma=1$, relation (MT2) gives $\alpha=1$, and
		$(\mathrm{MT4})_i$ gives $s_i=1$ for every $i$. Hence the old
		generators induce a homomorphism $G(P^+)\to G(M(P,w,S))$. Conversely,
		setting $\alpha=\beta=\gamma=1$ defines a homomorphism
		$G(M(P,w,S))\to G(P^+)$. The two maps are mutually inverse, proving (i).
		
		Now suppose that $w\ne1$ in $G(P)$. Put
		\[
		G_L=\bigl(G(P)*\langle\beta_L\rangle\bigr)*\langle\alpha\rangle,
		\qquad
		G_R=\langle\beta_R\rangle*\langle\gamma\rangle .
		\]
		In $G_L$, define
		\[
		\begin{aligned}
			u_0&=\beta_L,\\
			u_1&=\alpha^{-1}\beta_L\alpha,\\
			u_2&=\alpha^{-2}\beta_L^{-1}\alpha\beta_L\alpha^2,\\
			u_3&=\alpha^{-3}[w,\beta_L]\alpha^3,\\
			u_{3+i}&=\alpha^{-(i+3)}s_i\beta_L\alpha^{i+3},
			\qquad i=1,\dots,m'.
		\end{aligned}
		\]
		Apply Lemma~\ref{lem:branch-normal-form} with
		\[
		A=G(P)*\langle\beta_L\rangle,\qquad t=\alpha,
		\]
		and with the indices
		\[
		D=\{0,1,\ldots,m'+3\}.
		\]
		The elements $u_0,u_1,u_2,u_3$, and $u_{3+i}$ correspond respectively to the distinct indices $0,1,2,3$, and $i+3$. The element $u_2$ has form (Q$_2$), with $b_2=\beta_L$, while all the other displayed elements have form~(P$_d$) for their corresponding indices $d$.
		
		The element $\beta_L$ has infinite order. Since $w\ne1$, $[w,\beta_L]=w^{-1}\beta_L^{-1}w\beta_L$
		is cyclically reduced of syllable length four in
		$G(P)*\langle\beta_L\rangle$,
		and therefore has infinite order. 
		
		For each $i$, if $s_i=1$ in $G(P)$, then $s_i\beta_L=\beta_L$; otherwise, $s_i\beta_L$ is cyclically reduced of syllable length two. Thus $s_i\beta_L$ has infinite order in either case. All the hypotheses of Lemma~\ref{lem:branch-normal-form} are therefore satisfied, and it follows that
		\[
		u_0,u_1,\ldots,u_{m'+3}
		\]
		freely generate a free subgroup $A_L\le G_L$.
		
		In $G_R$, define
		\[
		\begin{aligned}
			v_0&=\beta_R,\\
			v_1&=\gamma^{-1}\beta_R^{-1}\gamma\beta_R\gamma,\\
			v_2&=\gamma^{-2}\beta_R^{-1}\gamma\beta_R\gamma^2,\\
			v_3&=\gamma^{-3}\beta_R\gamma^3,\\
			v_{3+i}&=\gamma^{-(i+3)}\beta_R\gamma^{i+3},
			\qquad i=1,\ldots,m'.
		\end{aligned}
		\]
		Apply Lemma~\ref{lem:branch-normal-form} with
		\[
		A=\langle\beta_R\rangle,\qquad t=\gamma,
		\]
		and the same index set $D=\{0,1,\ldots,m'+3\}$. The displayed elements have forms
		\[
		(\mathrm{P}_0),\ (\mathrm{Q}_1),\ (\mathrm{Q}_2),\ (\mathrm{P}_3),\ (\mathrm{P}_{i+3}),
		\]
		respectively, and their associated indices are pairwise distinct. Every coefficient required in Lemma~\ref{lem:branch-normal-form} is nontrivial and, where required, has infinite order. Hence
		$
		v_0,v_1,\ldots,v_{m'+3}
		$
		freely generate a free subgroup $A_R\le G_R$.
		
		The basis correspondence $u_j\mapsto v_j$ defines an isomorphism
		\(
		\phi\colon A_L\longrightarrow A_R.
		\)
		After the relation $u_0=v_0$ identifies $\beta_L$ with $\beta_R$, the
		presentation of the amalgamated free product
		$
		G_L*_{\phi}G_R
		$
		is precisely $M(P,w,S)$. The normal-form theorem for amalgamated free
		products gives injections of $G_L$ and $G_R$ into this group. Since
		$G(P)$ is a free factor of $G_L$, the natural map
		$
		G(P)\hookrightarrow G(M(P,w,S))
		$
		is injective, proving (ii). If $P^+$ presents the trivial group, (i) gives the yes-case,
		whereas (ii) gives a nontrivial subgroup in every no-case. This proves the
		final equivalence.
	\end{proof}
	The preceding theorem controls the no-case. The following relative
	Andrews--Curtis calculation supplies the yes-case control needed in
	the PL reduction.
	
	\begin{lemma}
		\label{lem:relative-AC}
		Let $P=\langle X\mid R\rangle$, let $S=(s_1,\dots,s_{m'})$, $R=(r_1,\ldots,r_m)$, and suppose
		that $w=1$ in $G(P)$. Then the relator tuple
		of $M(P,w,S)$ is Andrews--Curtis equivalent to
		\[
		(R,s_1,\dots,s_{m'},\alpha,\beta,\gamma).
		\]
	\end{lemma}
	
	\begin{proof}
		Denote the relators corresponding to (MT1), (MT2), (MT3), and
		$(\mathrm{MT4})_i$ by $b_1,b_2,b_3,c_i$. Put
		\[
		d=[w,\beta],
		\qquad
		C=\gamma^{-3}\beta\gamma^3.
		\]
		Thus
		\[
		\begin{aligned}
			b_1&=\alpha^{-1}\beta\alpha
			\left(\gamma^{-1}\beta^{-1}\gamma\beta\gamma\right)^{-1},\\
			b_2&=\alpha^{-2}\beta^{-1}\alpha\beta\alpha^2
			\left(\gamma^{-2}\beta^{-1}\gamma\beta\gamma^2\right)^{-1},\\
			b_3&=\alpha^{-3}d\alpha^3C^{-1},\\
			c_i&=\alpha^{-(i+3)}s_i\beta\alpha^{i+3}
			\left(\gamma^{-(i+3)}\beta\gamma^{i+3}\right)^{-1}.
		\end{aligned}
		\]
		
		Since $w=1$ in $G(P)$, the word $w$ lies in the normal closure of $R$
		inside $\Free(X)$. Therefore
		\[
		d=w^{-1}\beta^{-1}w\beta
		\]
		and $d^{-1}$ lie in the normal closure of $R$ inside
		$\Free(X,\alpha,\beta,\gamma)$. The same is true of
		\[
		U:=C\alpha^{-3}d^{-1}\alpha^3C^{-1}.
		\]
		By Lemma~\ref{lem:ambient-normal-closure}, using only the entries of $R$, replace
		$b_3$ by $b_3U$. Then
		\[
		\begin{aligned}
			b_3U
			&=\alpha^{-3}d\alpha^3C^{-1}
			C\alpha^{-3}d^{-1}\alpha^3C^{-1}\\
			&=\alpha^{-3}dd^{-1}\alpha^3C^{-1}
			=C^{-1}.
		\end{aligned}
		\]
		After inversion and conjugation by $\gamma^3$, this entry is $\beta$.
		
		Using the relator $\beta$, delete every occurrence of
		$\beta^{\pm1}$ from $b_1$. The result is
		\[
		\alpha^{-1}\alpha(\gamma^{-1}\gamma\gamma)^{-1}=\gamma^{-1},
		\]
		which gives the relator $\gamma$. Using $\beta$ and $\gamma$, delete
		their occurrences from $b_2$; the result is
		\[
		\alpha^{-2}\alpha\alpha^2=\alpha.
		\]
		Finally, using $\alpha,\beta,\gamma$, delete their occurrences from each
		$c_i$, leaving $s_i$. Every deletion is an application of
		Lemma~\ref{lem:ambient-normal-closure}. The resulting tuple is, up to
		permutation and inversion,
		\[
		(R,s_1,\dots,s_{m'},\alpha,\beta,\gamma),
		\]
		as required.
	\end{proof}
	
	\begin{lemma}
		\label{lem:AC-control}
		Let $P=\langle X\mid R\rangle$ be a balanced presentation on $n$
		generators, let $S=(s_1,s_2)$, and suppose that
		\[
		P^+=\langle X\mid R,s_1,s_2\rangle
		\]
		is $2$-AC-trivial. If $w=1$ in $G(P)$, then $M(P,w,S)$ is
		$2$-AC-trivial.
	\end{lemma}
	
	\begin{proof}
		By Lemma~\ref{lem:relative-AC}, the full relator tuple is AC-equivalent to
		\[
		(R,s_1,s_2,\alpha,\beta,\gamma).
		\]
		Perform the assumed AC moves on the first $n+2$ entries, leaving the last
		three fixed. The tuple becomes
		\[
		(x_1,\dots,x_n,1,1,\alpha,\beta,\gamma),
		\]
		and hence, after permutation,
		\[
		(x_1,\dots,x_n,\alpha,\beta,\gamma,1,1).
		\]
		The modified gadget has $n+3$ generators and $n+5$ relators in this case,
		so this is exactly $2$-AC-triviality.
	\end{proof}
	
	\section{The final presentation associated to a compressed word}
	\label{sec:final-presentation}
	
	Apply the modified Miller--Tancer construction to
	\[
	P=Q_A,
	\qquad
	w=y_{\mathrm{out}},
	\qquad
	S=(x_0,x_1),
	\]
	and define
	\[
	P_A:=M(Q_A,y_{\mathrm{out}},(x_0,x_1)).
	\]
	
	\begin{lemma}\label{lem:PA}
		The presentation $P_A$ has $\defp(P_A)=2$. Moreover, (i) if $\val(A)=1$ in $F$, then $P_A$ is $2$-AC-trivial;
		and (ii) if $\val(A)\ne1$ in $F$, then the natural map
		\(
		F\cong G(Q_A)\longrightarrow G(P_A)
		\)
		is injective.
	\end{lemma}
	
	\begin{proof}
		By Lemma~\ref{lem:QA}, $Q_A$ is balanced. If it has $n$ generators and
		$m$ relators, then $m=n$. The Miller--Tancer construction with two
		additional trivialising words has surplus
		\[
		m+2-n=2.
		\]
		
		If $\val(A)=1$ in $F$, then $y_{\mathrm{out}}=1$ in $G(Q_A)$. By
		Lemma~\ref{lem:QA}, $Q_A^+$ is
		$2$-AC-trivial. Lemma~\ref{lem:AC-control} implies that $P_A$ is
		$2$-AC-trivial.
		
		If $\val(A)\neq1$ in $F$, then $y_{\mathrm{out}}\neq1$ in $G(Q_A)$. By
		Theorem~\ref{thm:miller-tancer}, the natural map
		$F\cong G(Q_A)\to G(P_A)$ is injective.
	\end{proof}
	
	The construction $A\mapsto P_A$ is polynomial-time: each straight-line program
	gate adds one generator and one relator of length at most $3$, and the
	modified gadget adds three generators and five relators whose lengths are
	polynomial in the size of $A$. The exponentially long word $\val(A)$ is
	never expanded.
	
	\section{From a presentation to a triangulated closed
		\texorpdfstring{$4$}{4}-manifold}
	\label{sec:triangulation}
	
	We record the explicit polynomial-time passage from a finite presentation to
	the triangulated closed $4$-manifold used below. The triangulation step uses
	the algorithm of Casali--Cristofori \cite{CasaliCristofori} for converting Kirby diagrams into
	coloured triangulations.
	
	Let
	\[
	P=\langle z_1,\dots,z_n\mid \rho_1,\dots,\rho_r\rangle,
	\qquad r=n+k,
	\]
	be a finite presentation with $\defp(P)=k$, and set
	\[
	\|P\|:=n+r+\sum_{j=1}^{r}|\rho_j|,
	\]
	where $|\rho_j|$ denotes the length of the displayed input word.
	Thus every generator and every relator, including an empty relator,
	contributes to the input size.
	
	For a finite presentation $P$ as above,
	let $K(P)$ denote its standard presentation complex, with one vertex,
	one oriented $1$-cell labelled $z_i$ for each generator, and one
	$2$-cell attached along the loop representing $\rho_j$ for each
	relator. When the presentation is displayed explicitly, we also write
	\[
	K\langle z_1,\dots,z_n\mid \rho_1,\dots,\rho_r\rangle
	\]
	for this complex. 
	
	\subsection{The presentation handlebody}
	
	Construct a Kirby diagram $L(P)$ with one dotted unknot per generator, with all of these components mutually unlinked. If
	\[
	\rho_j=z_{i_1}^{\varepsilon_1}\cdots
	z_{i_\ell}^{\varepsilon_\ell},
	\qquad \varepsilon_s\in\{-1,1\},
	\]
	draw an attaching circle which passes through the dotted $1$-handle corresponding to
	$z_{i_s}$, with orientation determined by $\varepsilon_s$, in the order
	prescribed by the word. An empty relator is represented by an isolated
	unknot.
	
	Let $W^{(1)}(P)$ be the resulting $1$-handlebody. Give every
	relator component $C_j$ coefficient $0$ in the usual integer
	convention for Kirby diagrams. Here $0$ denotes the usual
	$0$-framing of $C_j$ as a knot in $S^3$; it is not, in general,
	the blackboard framing. Thus, if $w_j$ denotes the writhe of the
	drawn component, the prescribed framing differs from the blackboard
	framing by $-w_j$ full twists.
	
	Denote the resulting compact smooth $4$-dimensional handlebody,
	with its induced PL structure, by $W(P)$. It has one $0$-handle,
	$n$ $1$-handles, and $r$ $2$-handles.
	
	\begin{lemma}
		For every finite presentation $P$, the diagram $L(P)$
		can be constructed in polynomial time with $O(\|P\|^2)$ crossings.
		If $P$ has at least one generator and at least one relator, then
		$L(P)$ can moreover be put, with polynomial overhead, into
		the form required by the Casali--Cristofori construction.
	\end{lemma}
	
	\begin{proof}
		Place the dotted components in disjoint boxes and route each relator
		component through these boxes in word order. Each letter contributes a
		bounded local pattern, so the diagram has $O(\|P\|)$ elementary arcs
		and bends. Placing them on an $O(\|P\|)$-by-$O(\|P\|)$ grid with fixed
		crossing conventions gives $O(\|P\|^2)$ crossings. Let
		$s_{L}$ denote the resulting number of crossings.
		
		For each relator component $C_j$, let $w_j$ be its writhe. Since
		each self-crossing contributes to the writhe of exactly one component,
		\[
		\sum_{j=1}^{r}|w_j|\le s_{L}.
		\]
		The coefficient of $C_j$ is $c_j=0$. In the notation of
		\cite[Theorem~12]{CasaliCristofori}, the framing-preparation parameter~is
		\[
		\bar t_j=
		\begin{cases}
			|w_j|,&w_j\ne0,\\
			2,&w_j=0.
		\end{cases}
		\]
		Consequently,
		\[
		\sum_{j=1}^{r}\bar t_j
		=
		\sum_{w_j\ne0}|w_j|
		+2\#\{j:w_j=0\}
		\le s_{L}+2r
		=O(\|P\|^2).
		\]
		Thus the additional curls and the remaining local data required by
		Procedure~C have polynomial size and can be constructed in polynomial
		time. The resulting prepared diagram therefore still has polynomial
		size.
	\end{proof}
	
	\subsection{From the Kirby diagram to the double}
	
	For every finite presentation $P$, set
	\[
	X(P):=D(W(P))
	=W(P)\cup_{\partial W(P)}W(P).
	\]
	
	\begin{lemma}
		\label{lem:kirby-output}
		Let $P$ be a finite presentation with at least one generator.
		There is a polynomial-time algorithm which, given $P$, constructs a
		finite triangulation representing the closed PL $4$-manifold $X(P)$.
		Moreover,
		\[
		X(P)=D(W(P))
		\cong_{\PL}\partial(W(P)\times I).
		\]
		The number of top-dimensional simplices is bounded by a polynomial
		in $\|P\|$.
	\end{lemma}
	
	\begin{proof}
		If $L(P)$ has no framed components (that is, if the presentation has no relators), then $L(P)$ represents the boundary sum $\bs_n (S^1\times D^3)$ and can be triangulated directly (e.g.\ by layering). 
		
		Otherwise, apply Procedure C of Casali--Cristofori
		\cite[Proposition~6, Theorems~7 and~12]
		{CasaliCristofori} to $L(P)$ (we note that this procedure is implemented in the software \emph{Katie} \cite{Burke-software}).
		
		The output of the Casali--Cristofori procedure is an edge-coloured graph dual to a triangulation~$\widehat{\mathcal{T}}(P)$ of 
		\[
		\widehat{W}(P)=W(P)\cup_{\partial W(P)}\operatorname{Cone}(\partial W(P)),
		\]
		with a distinguished colour-$4$ vertex $v_4$ whose link is $\partial W(P)$. If $\partial W(P)\not\cong S^3$, then $\widehat{\mathcal{T}}(P)$ is an ideal triangulation. We retain this distinguished vertex even if its link is $S^3$. 
		
		Truncate the colour-$4$ corner in each $4$-simplex of $\widehat{\mathcal{T}}(P)$. Each resulting truncated $4$-simplex is a copy of $\Delta^3\times I$, which can be divided into four $4$-simplices using the standard staircase triangulation for a product. The unpaired tetrahedra assemble to form $\operatorname{lk}(v_4)\cong\partial W(P)$, and the resulting triangulation $\mathcal{T}(P)$ triangulates $W(P)$ (and has size linear in that of $\widehat{\mathcal{T}}(P)$). Note that if $\operatorname{lk}(v_4)\cong S^3$, truncating $v_4$ simply removes a $4$-ball attached to $W(P)$.
		
		Since the Kirby diagram has at least one dotted component and
		at least one framed component, its underlying link is not the
		standard diagram of the trivial knot. Hence
		\cite[Theorem~12]{CasaliCristofori} applies and shows that the graph
		produced by Procedure~C has order
		\[
		N
		=
		8s_{L}
		+4\sum_{j=1}^{r}\bar t_j
		\le
		12s_{L}+8r
		=
		O(\|P\|^2).
		\]
		Thus $\widehat{\mathcal T}(P)$ has $N$ simplices. Truncating the
		distinguished colour-$4$ corner and applying the
		staircase subdivision into four $4$-simplices produces at most $4N$ simplices in
		$\mathcal T(P)$.
		
		Now take two copies $\mathcal{T}_+$ and $\mathcal{T}_-$ of $\mathcal T(P)$, with one copy oppositely
		oriented with respect to the other. The boundary triangulations are combinatorially isomorphic. Glue each boundary facet of $\mathcal{T}_+$ to the corresponding boundary facet of $\mathcal{T}_-$ by the identity affine map. The resulting closed triangulation has polynomial
		size and triangulates
		\[
		D(W(P))=W(P)\cup_{\partial W(P)}W(P).
		\]
		Finally,
		\[
		D(W(P))\cong_{\PL}\partial(W(P)\times I),
		\]
		which proves the lemma.
	\end{proof}
	
	\section{Topology of the output}
	\label{sec:output-topology}
	
	In this section we first compute the fundamental group of $X(P)$. We then identify its PL type in the Andrews–Curtis-trivial case by constructing a spin $5$-dimensional thickening, transporting its spine under a $3$-deformation, and applying the uniqueness theorem for thickenings.
	
	\begin{lemma}
		\label{lem:pi1-double}
		For every finite presentation $P$,
		\[
		\pi_1(X(P))\cong G(P).
		\]
	\end{lemma}
	
	\begin{proof}
		The boundary of $W(P)$ is connected; choose a basepoint there. The handle
		CW complex gives
		\[
		\pi_1(W(P))\cong G(P).
		\]
		Moreover, the inclusion
		\[
		\partial W(P)\hookrightarrow W(P)
		\]
		is surjective on fundamental groups. Indeed, turning the handle decomposition
		upside down constructs $W(P)$, relative to its boundary, with handles of
		indices $2,3,4$, and these handles do not create fundamental-group
		generators.
		
		The Seifert--van Kampen theorem gives
		\[
		\pi_1(D(W(P)))
		\cong
		\pi_1(W(P))*_{\pi_1(\partial W(P))}\pi_1(W(P)).
		\]
		After identifying the two copies of $W(P)$, both edge maps are the same
		epimorphism. The pushout of two copies of a group along the same epimorphism
		is that group. Hence
		\[
		\pi_1(D(W(P)))\cong\pi_1(W(P))\cong G(P).
		\]
		Since $X(P):=D(W(P))$, the result follows.
	\end{proof}
	
	\begin{lemma}
		\label{lem:zero-w2-product}
		The 5-manifold $N(P)=W(P)\times I$ is an orientable 5-dimensional thickening of $K(P)$ with vanishing second Stiefel--Whitney class $w_2$. Explicitly, $K(P)$ admits a PL embedding in $\operatorname{int}(N(P))$ such that, identifying $K(P)$ with its image,
		\[N(P)\searrow K(P),\qquad w_2(N(P))=0.\]
	\end{lemma}
	
	\begin{proof}
		Crossing the handle decomposition of $W(P)$ with $I$ gives
		$N(P)$ a $5$-dimensional handle decomposition with only
		$0$-, $1$- and $2$-handles. Its $0/1$-handlebody collapses
		onto the wedge $V$ of the $1$-handle cores. Under this collapse,
		the attaching circle of the $j$-th $2$-handle maps to the
		cellular loop in $V$ spelling $\rho_j$.
		
		Adjoin each $2$-handle core via the mapping cylinder of its attaching loop. 
		In dimension $5$, PL general position allows
		these mapping cylinders and core discs to be chosen simultaneously.
		Their union with $V$ is a PL-embedded copy of $K(P)$, after
		subdivision. 
		
		The core-and-mapping-cylinder construction exhibits
		$N(P)$ as a regular neighbourhood of this copy, and hence gives a
		PL collapse
		\(
		N(P)\searrow K(P).
		\)
		It remains to determine the thickening class. The $0/1$-handlebody
		$W^{(1)}(P)$ is spin; choose the spin structure obtained by making
		the product choice across each $1$-handle. The spin-extension rule
		for dotted Kirby diagrams states that this spin structure extends
		over a $2$-handle precisely when the integer coefficient of its
		attaching circle is even
		\cite[\S 5.7]{GompfStipsicz}. Every $2$-handle of $W(P)$ has
		coefficient $0$, so the chosen spin structure extends over all of
		them. Hence
		\(
		w_2(W(P))=0
		\)
		and therefore
		\[
		w_2(N(P))
		=
		w_2(W(P)\times I)
		=
		\operatorname{pr}_{W(P)}^*w_2(W(P))
		=
		0.
		\]
		
		Finally, the collapse induces an isomorphism
		\[
		H^2(K(P);\mathbb Z/2)
		\xrightarrow{\;\cong\;}
		H^2(N(P);\mathbb Z/2).
		\]
		Hence the thickening class whose pullback is $w_2(N(P))$ is the
		zero class.
	\end{proof}
	
	\begin{theorem}[Spine transport under $3$-deformation]
		\label{thm:spine-transport}
		Let $N$ be a compact PL $5$-manifold with boundary and suppose that $N$
		collapses onto a finite $2$-complex
		$K\subset\operatorname{int}(N)$. If $K$ $3$-deforms to a finite
		$2$-complex $M$, then $M$ admits a PL embedding in
		$\operatorname{int}(N)$ such that $N$ collapses onto the embedded copy of
		$M$.
	\end{theorem}
	
	\begin{proof}
		Apply \cite[\S~2.1, Fact~(2)]{FKL} to the spine $K\subset\operatorname{int}(N)$ and the given 3-deformation from $K$ to $M$. It yields a PL embedding $j\colon M\hookrightarrow\operatorname{int}(N)$ whose image is a spine of $N$. Thus $N\searrow j(M)$, as required. 
	\end{proof}
	
	\begin{theorem}[Uniqueness of presentation thickenings]\label{thm:thickening-classification} 
		Let $K$ be a connected finite presentation complex. For every class
		\[
		w\in H^2(K;\mathbb Z/2),
		\]
		there exists a compact orientable smooth 5-manifold $N(K,w)$ containing a PL-embedded copy of $K$ in its interior and admitting a PL collapse
		\[
		N(K,w)\searrow K,
		\]
		such that the retraction $r:N(K,w)\to K$ furnished by this collapse satisfies
		\[
		r^*w=w_2(N(K,w)).
		\]
		Any two such thickenings for the same $K$ and $w$ have diffeomorphic underlying smooth manifolds. In particular, their underlying PL manifolds are PL-homeomorphic.
	\end{theorem}
	\begin{proof}
		This is the presentation-complex form of the thickening classification in \cite[Lemma~3.3]{HKT}; see also \cite{WallThickenings}. In the associated 5-dimensional handle construction, the attaching circles of the 2-handles are determined up to isotopy by the relators. The framing choices lie in $\pi_1(\mathrm{SO}(3))\cong\mathbb Z/2$ and determine a cellular 2-cochain representing the second Stiefel–Whitney class under the identification with $H^2(K;\mathbb Z/2)$. Every class $w$ can be realised in this way. Changing the representing cochain by a coboundary is realised by twisting the 1-handles and does not change the diffeomorphism type, giving the asserted uniqueness.
	\end{proof}
	
	\begin{lemma}
		\label{lem:AC-standard-double}
		Let
		\[
		P=\langle z_1,\dots,z_n\mid\rho_1,\dots,\rho_{n+k}\rangle
		\]
		be $k$-AC-trivial. Then
		\[ X(P)\cong_{\mathrm{PL}}\partial(W(P)\times I)\cong_{\PL}\cs_k(S^2\times S^2), \]
		where the $k=0$ case means $S^4$.
	\end{lemma}
	
	\begin{proof}
		The geometric interpretation of Andrews--Curtis moves gives a
		$3$-deformation between the corresponding presentation complexes:
		inversion reverses a $2$-cell, multiplication is a $2$-cell slide, and
		conjugation changes its whisker
		\cite{AndrewsCurtis,Cohen,HogAngeloniMetzler}. Hence $K(P)$
		$3$-deforms to
		\[
		K\langle z_1,\dots,z_n\mid z_1,\dots,z_n,1,\dots,1\rangle .
		\]
		The first $n$ $2$-cells collapse the $1$-cells, while the $k$ empty
		relators remain as $2$-spheres. Thus
		\[
		K(P)\quad 3\text{-deforms to}\quad \wedgek{k}.
		\]
		
		By Lemma~\ref{lem:zero-w2-product}, $N(P)=W(P)\times I$ collapses onto
		$K(P)$ and $w_2(N(P))=0$. Theorem~\ref{thm:spine-transport} produces an
		embedded copy of $\wedgek{k}$ onto which $N(P)$ collapses. If
		\[
		r'\colon N(P)\longrightarrow\wedgek{k}
		\]
		is the resulting collapse, then $(r')^*$ is an isomorphism on
		$H^2(-;\mathbb Z/2)$. The thickening class relative to this new spine is
		therefore zero, since its pullback is $w_2(N(P))=0$.
		
		The boundary connected sum
		\[
		\bs_k(S^2\times D^3)
		\]
		is also an orientable $5$-dimensional thickening of $\wedgek{k}$ with vanishing $w_2$.
		The wedge is itself a presentation complex, with no generators and $k$
		empty relators, so Theorem~\ref{thm:thickening-classification} applies and
		gives
		\[
		W(P)\times I\cong_{\PL}\bs_k(S^2\times D^3).
		\]
		Taking boundaries yields
		\[
		X(P)=\partial(W(P)\times I)
		\cong_{\PL}
		\partial\bigl(\bs_k(S^2\times D^3)\bigr)
		\cong_{\PL}
		\cs_k(S^2\times S^2).\qedhere
		\]
	\end{proof}
	
	\section{Proof of the main theorem}
	\label{sec:proof-of-theorem}
	
	Fix a closed connected PL $4$-manifold $Z$ and a finite
	triangulation $\mathcal Z$ of $Z$. Let $A$ be an instance of
	$\CWP(F)$. Construct $P_A$ as in
	Section~\ref{sec:final-presentation}. Apply
	Lemma~\ref{lem:kirby-output} to $P_A$, and let $\mathcal T_A$ be the
	resulting finite triangulation. Set
	\(
	X_A:=|\mathcal T_A|.
	\)
	Then
	\[
	X_A\cong_{\PL}X(P_A)=D(W(P_A)),
	\]
	and $\mathcal T_A$ can be constructed in polynomial time.
	
	We next form the connected sum with $Z$.
	Let $\mathcal T$ be either $\mathcal T_A$ or
	$\mathcal Z$, choose a paired tetrahedral facet of one of its
	simplices and unglue this pair of facets. Insert the fixed block
	\[
	\bigl(\Delta^3\times I\bigr)
	\mathbin{\cup}_{\partial\Delta^3\times I}
	\bigl(\Delta^3\times I\bigr)
	\cong_{\PL}S^3\times I.
	\]
	Here each tetrahedral prism $\Delta^3\times I$ carries its standard
	staircase triangulation into four $4$-simplices, and corresponding vertical
	faces of the two prisms are identified. Attach the two tetrahedral ends in
	one boundary component of this block to the two sides of the opened facet
	pairing, transporting the vertex labels by the original gluing map. Leave
	the two tetrahedral ends in the other boundary component unglued; with
	their induced face pairings, they form a standard two-tetrahedron
	triangulation of $S^3$. 
	The resulting triangulation represents $|\mathcal T|$ with an open PL $4$-ball removed, while retaining every original simplex and adding exactly eight new simplices.
	
	Perform this puncture in both $\mathcal T_A$ and $\mathcal Z$, and
	identify the two resulting two-tetrahedron boundary $3$-spheres by a
	fixed simplicial isomorphism. Denote the resulting
	triangulation by $\mathcal T_{A,Z}$, and set
	\[
	X_{A,Z}:=|\mathcal T_{A,Z}|.
	\]
	Then
	\(
	X_{A,Z}\cong_{\PL}X_A\cs Z.
	\)
	Neither puncture removes an original simplex, and the final boundary
	identifications add none. Hence
	\[
	f_4(\mathcal T_{A,Z})
	=f_4(\mathcal T_A)+f_4(\mathcal Z)+8+8
	=f_4(\mathcal T_A)+f_4(\mathcal Z)+16
	=f_4(\mathcal T_A)+O_Z(1).
	\]
	Since $\mathcal Z$ is fixed and each inserted block has fixed size,
	$\mathcal T_{A,Z}$ can be constructed in polynomial time and always
	represents a closed PL $4$-manifold.
	
	If $\val(A)=1$ in $F$, then Lemma~\ref{lem:PA} says that $P_A$ is
	$2$-AC-trivial. Lemma~\ref{lem:AC-standard-double} gives
	\[
	X_A\cong_{\PL}\cs_{2}(S^2\times S^2),
	\]
	and hence
	\[
	X_{A,Z}\cong_{\PL}Z\cs_{2}(S^2\times S^2).
	\]
	
	Suppose that $\val(A)\neq1$ in $F$, and put
	\[
	G_A:=G(P_A),
	\qquad
	H:=\pi_1(Z).
	\]
	Lemma~\ref{lem:PA} gives an embedding $F\hookrightarrow G_A$, so
	$G_A$ is nontrivial. Lemma~\ref{lem:pi1-double} and the
	Seifert--van Kampen theorem give
	\[
	\pi_1(X_{A,Z})\cong G_A*H.
	\]
	On the other hand, $\cs_2(S^2\times S^2)$ is simply connected, and
	therefore
	\[
	\pi_1\left(Z\cs_2(S^2\times S^2)\right)\cong H.
	\]
	The groups $G_A$ and $H$ are finitely generated. By the
	Grushko--Neumann theorem,
	\[
	d(G_A*H)=d(G_A)+d(H)>d(H).
	\]
	Thus the two fundamental groups are not isomorphic, and so
	\[
	X_{A,Z}\not\cong_{\PL}
	Z\cs_2(S^2\times S^2).
	\]
	We have therefore established that 
	\[
	\val(A)=1\text{ in }F
	\quad\Longleftrightarrow\quad
	|\mathcal T_{A,Z}|\cong_{\PL}
	Z\cs_2(S^2\times S^2).
	\]
	
	Since $\CWP(F)$ is PSPACE-complete \cite[Corollary~56]{BFLaW}, the
	fixed-target recognition problem is PSPACE-hard. Since
	$\mathrm{NP}\subseteq\mathrm{PSPACE}$, it is also NP-hard.
	
	Finally, fix $n\geq2$ and take
	\[
	Z=\cs_{n-2}(S^2\times S^2),
	\]
	where the case $n=2$ means $Z=S^4$. The target in the preceding
	reduction is then $\cs_n(S^2\times S^2)$. This proves every assertion
	of Theorem~\ref{thm:main}.
	
	\section{Remarks on lower simply connected targets}
	\label{sec:final-remarks}
	
	Theorem~\ref{thm:main} gives PSPACE-hardness for every target obtained by
	adding two $S^2\times S^2$ summands to a fixed closed connected PL
	$4$-manifold. Within the family $\cs_n(S^2\times S^2)$, it leaves open
	the two lower members $S^4$ and $S^2\times S^2$. We now explain why
	these targets are not reached by the direct presentation--thickening
	construction used above.
	
	This direct pipeline imposes two kinds of constraints on its output. The
	algebraic stage must preserve enough of the source group to distinguish
	no-instances, while producing a presentation with a controlled
	$3$-deformation type in the yes-case. The topological stage then converts
	that $3$-deformation type into the connected-sum decomposition of the
	target.
	
	More precisely, suppose that $P$ has $n$ generators and $m$
	relators, and that its presentation complex $3$-deforms to
	\[
	Y_{p,q}:=\bigvee_p S^1\vee\bigvee_q S^2,\quad\quad\quad p,q\geq 0.
	\]
	By Lemma~\ref{lem:zero-w2-product} and
	Theorem~\ref{thm:spine-transport}, $Y_{p,q}$ is realised as a spine of $N(P)$ with $w_2=0$. 
	The thickening classification,
	Theorem~\ref{thm:thickening-classification}, therefore gives
	\[
	N(P)\cong_{\PL}
	\bs_p(S^1\times D^4)
	\bs_q(S^2\times D^3).
	\]
	Taking boundaries yields
	\begin{equation}
		\label{eq:relative-wedge-target}
		X(P)\cong_{\PL}
		\cs_{p}(S^1\times S^3)
		\cs_{q}(S^2\times S^2),
	\end{equation}
	with the corresponding summand omitted when $p=0$ or $q=0$.
	Moreover,
	\[
	1-n+m
	=
	\chi(K(P))
	=
	\chi(Y_{p,q})
	=
	1-p+q,
	\]
	and hence
	\begin{equation}
		\label{eq:surplus-wedge}
		\defp(P)=m-n=q-p.
	\end{equation}
	Thus the relator surplus records the difference between the numbers of
	$S^2\times S^2$ and $S^1\times S^3$ summands. 
	For the particular algebraic input used here, normal generation gives a
	further restriction. Apply the modified Miller--Tancer construction to
	$Q_A$ with a list
	\[
	S=(s_1,\dots,s_{m'}).
	\]
	Since $Q_A$ is balanced, the resulting presentation has surplus
	$m'$. When $\val(A)=1$ in $F$, Theorem~\ref{thm:miller-tancer} identifies
	its group with
	\[
	F\big/
	\left\langle\!\left\langle
	s_1,\dots,s_{m'}
	\right\rangle\!\right\rangle .
	\]
	The standard presentation of $F$ shows that
	\[
	F_{\mathrm{ab}}\cong\mathbb Z^2.
	\]
	Consequently, the quotient of $F$ by the normal closure of fewer than
	two elements cannot be trivial---after abelianisation, at most one
	relation has been imposed on $\mathbb Z^2$. Within the present choice
	of source group and algebraic gadget, a simply connected yes-case
	therefore requires at least two auxiliary words. In view of
	\eqref{eq:surplus-wedge}, the first simply connected target accessible
	through a wedge spine is correspondingly
	$
	\cs_2(S^2\times S^2).
	$
	
	This explains one obstruction to obtaining $S^2\times S^2$---such a
	target would require a surplus-one presentation whose yes-case spine
	$3$-deforms to $S^2$, whereas a single auxiliary word cannot
	normally generate $F$. Reaching $S^4$ would require a balanced
	output whose yes-case spine admits a $3$-deformation to a point. It is not known whether every balanced presentation
	of the trivial group has a presentation complex that $3$-deforms to a point. 
	This is an Andrews--Curtis-type requirement rather than a consequence of the group being trivial.
	
	These constraints indicate several possible routes towards lower
	targets---a source group requiring fewer normal generators, an algebraic gadget
	which decouples the relator surplus from the number of trivialising
	words, or a geometric operation which removes a summand while
	preserving the fundamental-group obstruction in the no-case.
		
	\bibliographystyle{plainurl}
	\bibliography{references}
	
\end{document}